\documentclass[reqno]{amsart}

\usepackage[T2A]{fontenc}
\usepackage[utf8]{inputenc}

\def\oo#1{\mathbin{{}_{(#1)}}}

\usepackage{amssymb,amsmath,amsthm}

\newtheorem{proposition}{Proposition}
\newtheorem{theorem}{Theorem}
\newtheorem{lemma}{Lemma}
\newtheorem{corollary}{Corollary}

\theoremstyle{definition}

\newtheorem{example}{Example}
\newtheorem{definition}{Definition}
\newtheorem{remark}{Remark}

\begin{document}

\title{Gr\"obner--Shirshov bases method for vertex algebras}

\thanks{The work was supported by the Russian Science Foundation (project 21-11-00286)}
\author{R.~A.~Kozlov$^{1,2)}$, P.~S.~Kolesnikov$^{1)}$}

\address{$^{1)}$Sobolev Institute of Mathematics SB RAS}
\address{$^{2)}$Novosibirsk State University}

\begin{abstract}
In the paper we show how to apply the Gr\"obner--Shirshov bases (GSB) method for modules over an associative algebra
to the study of vertex algebras defined by generators and relations. We compute GSBs for a series of vertex algebras
and study the problem of embedding of a left-symmetric 
algebra into a vertex one preserving the normally ordered product.
\end{abstract}

\maketitle

\noindent
\smash{\raise 10pt\rightline{\sl 
Dedicated to the memory of Aleksnder Vasilievich Mikhalev}}

The notion of a vertex algebra in an explicit algebraic form was proposed in~\cite{Borcherds}, where vertex algebras were applied for constructing a representation of the largest 
sporadic simple group (Monster).
The language of vertex algebras is commonly said to be 
a formal way to describe the properties of operator product expansion (OPE) of chiral fields in 2-dimensional conformal field theory (CFT, \cite{BPZ}) in physics.
In this way, a bulky definition of a vertex algebra appears 
that is widely presented in literature 
(see, e.g., \cite{FLM, F-BZvi}).
On the other hand, a categorial approach to the notion of a vertex algebra developed in~\cite{BD}
(see also \cite{BK-Field, BK-etal})
shows that from the algebraic point of view 
vertex algebras are close to the well-known classes of Lie and Poisson algebras.

One of most common ways to define an algebraic system in a
given class is to fix its generators and relations between them.
This approach works well in varieties of algebraic systems, like  (semi)groups, associative or Lie algebras, etc., since 
in this case any free algebra belongs to the considered class.

Although vertex algebras do not form a variety in the classical sense, it is possible to construct a universal object in this category for the class of vertex algebras generated by a given set of elements with a fixed dominating condition on their mutual locality function~\cite{Roit1999}.
Hence, there is an option to define a vertex algebra by means of defining relations. In this way, the first priority is to solve 
the word problem, that is, to find a linear basis of such a vertex algebra.
A standard technique for studying a word problem is the Gr\"obner--Shirshov bases method and its variations (see, e.g., \cite{BFKK,BC-book}).
The same method may be applied to vertex algebras as we show in this note.

Throughout the paper,  $\mathbb Z_+$ denotes the set of 
non-negative integers, all linear spaces are over a 
field $\Bbbk $ of characteristic zero.

\section{Conformal and vertex algebras}

Recall that a {\em Lie conformal algebra} \cite{KacVA_beginners}
is a linear space $L$ equipped with
 a linear operator $T: L\to L$ and
 with a family of
 $n$-products $(\cdot \oo{n} \cdot)$, $n\in \mathbb Z_+$,
satisfying a series of axioms. The latter can
be stated in a compact form by means of the
generating function of the sequence
$\{(x\oo{n} y)\}_{n\in \mathbb Z_+}$ for $x,y\in L$.
Namely, for every $x,y\in L$ the expression
\[
(x\oo\lambda y) = \sum\limits_{n\in \mathbb Z_+}
\dfrac{\lambda ^n}{n!} \otimes _{\Bbbk [T]} (x\oo n y)
\]
(called the {\em $\lambda $-bracket} of $x$, $y$)
should be a polynomial in the formal variable $\lambda $, i.e.,
$(x\oo\lambda y) \in \Bbbk [T,\lambda ]\otimes_{\Bbbk [T]} L \cong L[\lambda ]$.
Given $x,y \in L$, the degree of $(x\oo\lambda y)$ with respect to $\lambda $
gives rise to the {\em locality function} $N_L(x,y) \in \mathbb Z_+$:
\[
 N_L(x,y) = \begin{cases}
             0, & (x\oo\lambda y) = 0, \\
             \deg_\lambda (x\oo\lambda y) +1, & \text{otherwise}.
            \end{cases}
\]
In other words, $N_L(x,y)$ is the smallest $N\in \mathbb Z_+$
such that $(x\oo n y)=0 $ for all $n\ge N$.

Other axioms of a Lie conformal algebra
include {\em sesquilinearity}
\[
(Tx\oo\lambda y) = -\lambda (x\oo\lambda y),
\quad
(x\oo\lambda Ty) = (T+\lambda)(x\oo\lambda y),
\]
{\em skew symmetry}
\[
(x\oo\lambda y ) + (y\oo{T+\lambda } x) = 0,
\]
and the conformal version of {\em Jacobi identity}
\begin{equation}\label{eq:JacobiConf}
(x\oo\lambda (y\oo\mu z))
-
 (y\oo\mu (x\oo\lambda z))
 = ((x\oo\lambda y)\oo{\lambda+\mu } z).
\end{equation}

To define a Lie conformal superalgebra it is enough
to claim that $L=L_0\oplus L_1$ is a $\mathbb Z_2$-graded
$\Bbbk [T]$-module, the $n$-products preserve the grading,
and the axioms of skew symmetry along with the Jacobi identity
are modified in accordance with the Kaplansky rule (see, e.g.,
\cite{MP2010}). In particular, the Jacobi identity turns into
\[
(x\oo\lambda (y\oo\mu z))
-
(-1)^{|x||y|} (y\oo\mu (x\oo\lambda z))
 = ((x\oo\lambda y)\oo{\lambda+\mu } z).
\]
Hereinafter, $|x|\in \{0,1\}$ stands for the parity of a homogeneous element $x\in L_0\cup L_1$.

A series of examples is provided by quadratic Lie conformal superalgebras \cite{Xu2000}.
Let us state their construction in a particular case.
Suppose $V$ is a {\em Novikov superalgebra}, i.e., 
a $\mathbb Z_2$-graded linear space
equipped with a bilinear product $(\cdot \circ \cdot)$ such that
\[
 (u\circ v)\circ w - u\circ (v\circ w) = (-1)^{|u||v|} ((v\circ u)\circ w - v\circ (u\circ w)),
\]
\[
 (u\circ v)\circ w = (-1)^{|v||w|}(u\circ w)\circ v,
\]
for all homogeneous $u,v,w\in V$.
Then the free $\Bbbk [T]$-module $L(V) = \Bbbk [T]\otimes V$
is a Lie conformal superalgebra relative to the operation given by
\[
 (1\otimes u)\oo\lambda (1\otimes v) =(-1)^{|u||v|}T\otimes (v\circ u) +\lambda \otimes (u\circ v
 + (-1)^{|u||v|}v\circ u),
 \quad
 u,v\in V.
\]

\begin{example}\label{exmp:VirConformal}
 Let $V=\Bbbk v$ be a 1-dimensional algebra with the operation $v\circ v = v$.
 Then $L(V)$ is the {\em Virasoro conformal algebra}, its structure is completely defined
 by the ``conformal square'' of $1\otimes v = v$:
 \[
  (v\oo\lambda v ) = (T+2\lambda )v.
 \]
\end{example}

\begin{example}
Let $V$ be the 3-dimensional Novikov algebra
 $V=\Bbbk v +\Bbbk u + \Bbbk w$ with multiplication
 $v\circ v = v$, $v\circ u = \dfrac{1}{2}u$, $u\circ v = u$, $u\circ u=w$, $w\circ v=w$,
 other products are zero.
The corresponding Lie conformal algebra $L(V)$ is known as the  {\em Schr\"odinger--Virasoro conformal
algebra}.
\end{example}

A linear space $V$ equipped with a bilinear product
denoted $x.y$, $x,y\in V$, is said to be a left-symmertic
(pre-Lie) algebra, if
\[
(x.y).z - x.(y.z) =(y.x).z - y.(x.z)
\]
for all $x,y,z\in V$.
The definition of a left-symmetric superalgebra
can be derived in the ordinary way.
If $1$ is the identity element in a left-symmetric
(super)algebra $V$ with a (even) derivation $T$ then
\[
x.1 = 1.x = x, \quad T(x.y) = T(x).y + x.T(y), \quad T(1)=0,
\]
for all $x,y\in V$.

A vertex algebra structure may be considered as a composite
of a Lie conformal algebra and a left-symmetric algebra with a derivation.
It is somehow similar
to the structure of a Poisson algebra which is a composite of commutative
and Lie algebras.

\begin{definition}[see \cite{BK-Field}]
A linear space $V$ equipped with a linear operator $T$,
a binary operation
$(x,y)\mapsto x.y$, and a $\lambda $-bracket
$(x,y)\mapsto (x\oo\lambda y)$, $x,y\in V$,
is said to be a {\em vertex algebra} if
\begin{itemize}
 \item[{\rm (V1)}] $(V, {.})$ is a unital left-symmetric algebra with a derivation $T$;
 \item[{\rm (V2)}] $V$ with respect to $T$ and $(\cdot \oo\lambda \cdot ) $
 is a Lie conformal algebra;
 \item[{\rm (V3)}] The following identities hold for all $x,y,z\in V$:
\begin{gather}
x.y -  y.x = \int\limits_{-T}^0 (x\oo\lambda y)\, d\lambda
\label{eq:VertexComm} \\
(x\oo\lambda y.z) = (x\oo\lambda y).z + y.(x\oo\lambda z)
+\int\limits_0^\lambda ((x\oo\lambda y)\oo\mu z)\, d\mu .
\label{eq:WickIdent}
\end{gather}
\end{itemize}
The definition of a vertex superalgebra can be easily derived
by means of the Kaplansky sign rule.
\end{definition}

The ordinary definition of a vertex algebra (see, e.g., \cite{FLM, F-BZvi})
via vertex operators $Y(\cdot, z)$, operator-valued series in a formal variable $z$,
corresponds to the following presentation:
\[
Y(x,z)y = \sum\limits_{n\in \mathbb Z_+} (x\oo n y)z^{-n-1}
+ \sum\limits_{s\in \mathbb Z_+} \dfrac{1}{s!}(T^sx.y) z^s,
\]
for $x,y\in V$. The expression in the right-hand side
is a Lawrent formal series in $z$ with coefficients from~$V$.
The identity element of $(V, {.})$ is denoted $\mathbf 1$
(or $|0\rangle $), it is known
as the {\em vacuum vector} of~$V$. The left-symmetric operation
$(\cdot .\cdot )$ is represented in this context by the normally ordered
product of vertex operators:
\[
 Y(x.y, z) = :Y(x,z)Y(y,z):, \quad x,y\in V
\]
(see, e.g., \cite{FLM, F-BZvi}).

Suppose $X$ is a nonempty set. The class of all vertex algebras $V$ generated by the set $X$
does not contain universal object since there is no upper bound
for the locality function $N_V$ on $X\times X$. However,
if we restrict the locality on the generators then such a universal
object may be constructed \cite{Roit1999}.

Let us fix a function $N:X\times X\to \mathbb Z_+$ and define
the Lie algebra $\mathcal L(X,N)$ generated by the set $X\times \mathbb Z$
presented as
$\{x(n) \mid x\in X, n\in \mathbb Z \} $
modulo the following defining relations:
\begin{equation}\label{eq:FreeLoc}
 \sum\limits_{s=0}^{N(x,y)} (-1)^s \binom {N(x,y)}{s} [x(n-s),y(m+s)] = 0,\quad x,y\in X,\ n,m\in \mathbb Z.
\end{equation}
The universal enveloping associative algebra $U(\mathcal L(X,N))$
is also generated (as an associative algebra) by $X\times \mathbb Z$ relative to the commutator relations \eqref{eq:FreeLoc}. Denote by $\mathcal A(X,N)$
the associative algebra $\Bbbk [T]\ltimes U(\mathcal L(X,N))$, where
$T$ acts on the universal envelope as a derivation $x(n)\mapsto -nx(n-1)$,
$x\in X$, $n\in \mathbb Z$.
Namely, $\mathcal A(X,N)$ as an associative algebra is generated by the set $(X\times \mathbb Z)\cup \{T\}$
modulo the defining relations \eqref{eq:FreeLoc} and
\[
 Tx(n) - x(n)T = -nx(n-1),\quad x\in X,\ n\in \mathbb Z.
\]

Let $\mathrm{Vert}(X,N)$ stand for the left $\mathcal A(X,N)$-module
generated by a single element~$\mathbf 1$ relative to the defining
relations
\begin{equation}\label{eq:Vacuum}
T\mathbf 1=0,\quad  x(n)\mathbf 1 = 0, \quad x\in X,\ n\in \mathbb Z_+.
\end{equation}
Then $\mathrm{Vert} (X,N)$ is the universal object in the class of all vertex
algebras $V$ generated by $X$ such that $N_V(x,y)\le N(x,y)$ for all $x,y\in X$.
The generators $x\in X$ correspond to the elements $x(-1)\mathbf 1 \in \mathrm{Vert}(X,N)$,
and, more general, if $a\in \mathrm{Vert}(X,N)$, $x\in X$, $n\in \mathbb Z$ then
\begin{equation}\label{eq:CoeffOperator}
 x(n)a = \begin{cases}
          (x\oo{n} a), & n\in \mathbb Z_+, \\
          x.a, & n=-1, \\
          -\dfrac{1}{n+1} ( T(x(n+1)a) - x(n+1)Ta), & n\le -2.
         \end{cases}
\end{equation}
For example, if $a= x(1)y(-1)\mathbf 1$ and $b=z(-1)\mathbf 1$
then
$a.b = (x\oo 1 y).z$. The latter may be calculated via \eqref{eq:WickIdent}:
compare the coefficients at $\lambda ^1$ in left- and right-hand sides
of \eqref{eq:WickIdent} to get
\[
 (x\oo1 (y.z)) =  (x\oo1 y).z + y.(x\oo 1 z ) + ((x\oo 0 y)\oo 0 z).
\]
The last summand transforms by \eqref{eq:JacobiConf}, so we obtain
\[
 (x\oo1 y).z = (x\oo1 (y.z)) - y.(x\oo 1 z ) - (x\oo 0 (y\oo 0 z)) + (y\oo 0 (x\oo 0 z)),
\]
i.e.,
\[
 a.b = [x(1),y(-1)]z(-1)\mathbf 1 - [x(0),y(0)]z(-1)\mathbf 1.
\]

In general, every formal expression in terms of vertex algebra operations
on the elements of $X$ may be rewritten by means
of (V1)--(V3) as a $\Bbbk [T]$-linear
combination of right-normed words in $X$ 
that can be easily represented
in the $\mathcal A(X,N)$-module $\mathrm{Vert}(X,N)$.

Indeed, assume $V$ is a vertex algebra generated by a set $X$.
Then define $V(n)$, $n\in \mathbb Z_+$, to be the subspace of $V$
spanned by all those monomials in $X$ that contain no more than
$n$ operations $(\cdot .\cdot)$. Then $V(0)\subseteq V(1)\subseteq \dots $,
$V(n).V(m)\subseteq V(n+m+1)$,
$(V(n)\oo{s} V(m)) \subseteq V(n+m)$,
and the graded space $\bigoplus_{n\ge 0} V(n)/V(n-1) $
is a Poisson vertex algebra in the sense of \cite{BSK}.
(Note that \eqref{eq:VertexComm} implies the left-symmetric operation
is commutative on the graded space, thus associative.)
As in the case of ``ordinary'' Poisson algebras, an expression
in a Poisson vertex algebra may be written in a right-normed form.
Hence, by induction on $n\ge 0$, every expression in $V(n)$
may be also written in a right-normed form with respect to
vertex algebra operations.
The base of induction is $n=0$, where the claim follows from \eqref{eq:JacobiConf}.

The same observations
lead us to

\begin{proposition}
Let $X$ be a set, and let $N$ be a fixed locality function on $X\times X$.
Then a subset $I\subseteq \mathrm{Vert}(X,N)$ is an ideal of the vertex algebra
if and only if $I$ is an $\mathcal A(X,N)$-submodule of
 $\mathrm{Vert}(X,N)$.
\end{proposition}

\begin{proof}
 The ``only if'' part is obvious from \eqref{eq:CoeffOperator}:
if $I$ is an ideal of the vertex algebra
 $\mathrm{Vert}(X,N)$
 then it is closed under the action of $x(n)$, $x\in X$, $n\in \mathbb Z$.

 For the ``if'' part, suppose $I$ is an $\mathcal A(X,N)$-submodule of
 $\mathrm{Vert}(X,N)$ and $a\in I$.
Let $X'=X\cup \{a\}$ be the extended set of generators, and let
$N'$ be a locality function on $X'\times X'$ extending $N$ in such
a way that
$N'(x,a)$, $N'(a,x)$, $N'(a,a)$ are given by the corresponding localities
in $\mathrm {Vert}(X,N)$, or their estimates according to the Dong Lemma.
The initial vertex algebra $\mathrm{Vert}(X,N) $ is a homomorphic image of
$V' = \mathrm{Vert}(X',N')$.

Then for every $f\in \mathrm{Vert }(X,N)$ the elements
$f.a$, $a.f$, $(f\oo s a)$, $(a\oo s f)$
may be considered as expressions in~$V'$.
As it was shown above, every element from $V'$ may be presented
as a $\Bbbk [T]$-linear combination of right-normed words
in $X'$. Consider such a word
\[
w = x_1(n_1)\dots x_i(n_i)a(m)x_{i+1}(n_{i+1})\dots x_k(n_k)\mathbf 1, \quad m,n_i\in \mathbb Z.
\]
If $i=k$ then, obviously, $w$ is obtained from $a$ by the action of $\mathcal A(X,N)$.
If $i<k$ then look at the pair $a(m)x_{i+1}(n_{i+1})$, denote it by $a(m)x(n)$, and let $w'$
stand for the ``tail'' of $w$. In the case when $m,n\ge 0$ we may apply \eqref{eq:JacobiConf}
to express $a(m)x(n)w' = x(n)a(m)w' - \sum\limits_{s\ge 0} \binom{n}{s} (x\oo s a)(n+m-s) w'$.
The first summand contains $a(m)w'$ with shorter ``tail'', so its image in $V$ belongs to $I$
by inductive reasons. The second group of summands also contains $(x\oo s a)=x(s)a \in I$ with a shorter ``tail'' $w'$.
Hence, $w\in I$. In the case when $m\ge 0$ and $n<0$ we may apply \eqref{eq:WickIdent} and \eqref{eq:VertexComm}
to express $a(m)x(n)w'$ in a similar form, as a $\mathcal A(X,N)$-linear combination of
$a_j(m_j)w_j'$ with $a_j\in I$,
with shorter ``tails'' $w_j'$. In the case when $m<0$ and $n\ge 0$, one may again apply
\eqref{eq:WickIdent} along with conformal skew symmetry to get $w\in I$.
Finally, if $n,m<0$ then left symmetry of the product $(\cdot.\cdot)$ along with \eqref{eq:VertexComm}
also leads us to the conclusion $w\in I$.
\end{proof}

Therefore, in order to present a vertex algebra $V$ via generators $X$ (with a locality
function $N$ on $X\times X$) and relations $R$
one has to consider the quotient $\mathrm{Vert}(X,N\mid R)$
of the free 1-generated module over $\mathcal A(X,N)$ relative to
\eqref{eq:Vacuum} and the relations $R$ represented via \eqref{eq:CoeffOperator}.

\begin{example}\label{exmp:WeylVertex}
 Let $X=\{x,y\}$, $N(x,x)=N(y,y)=0$, $N(x,y)=N(y,x)=1$, $R=\{(x\oo 0 y)-\mathbf 1\}$.
Then
\[
 \mathcal A(X,N) = \Bbbk \langle T, x(n),y(m) \mid n,m\in \mathbb Z\rangle /(S),
\]
where
$S$ consists of
\[
\begin{gathered}{}
[T,a(n)]= -na(n-1),\quad a\in \{x,y\}, \\
[x(n),x(m)]=[y(n),y(m)]=
 [x(n),y(m)]-[x(n-1),y(m+1)] =0,
 \quad n,m\in \mathbb Z.
\end{gathered}
\]
The vertex algebra
$W = \mathrm{Vert}(X,N\mid R)$
is the quotient of the $\mathcal A(X,N)$-module generated by the vaccum vector
$\mathbf 1$ relative to the relations \eqref{eq:Vacuum} and
\[
 x(0)y(-1)\mathbf 1  = \mathbf 1.
\]
Below we show that $W$ is exactly the Weyl vertex algebra \cite{AdamPedic19}.
\end{example}

In order to study the structure of a module over an associative algebra,
one may apply the Gr\"obner--Shirshov bases technique for modules \cite{KangLee2000}
which is briefly described in the next section.

\section{Gr\"obner--Shirshov bases for modules: application to vertex algebras}

The Gr\"obner--Shirshov bases (GSB) technique for associative algebras is very well
known and described
in a series of papers from various points of view
(see, e.g., \cite{BokChen-Bull, DotsTamaroff}).

Namely, let $\mathcal X$ be a set. Then the free associative algebra
$\Bbbk \langle \mathcal X\rangle $ generated by $\mathcal X$ is equal to $\Bbbk \mathcal X^*$,
the linear span of all words in~$\mathcal X$.

Assume $S\subset \Bbbk \langle \mathcal X\rangle $
is a set of nonzero polynomials presented in the form $f = \bar f - r(f)$,
where $\bar f$ is a selected word with identity coefficient (principal part of~$f$),
and $r(f)$
denotes the sum of all other terms in~$f$.
Then $S$ gives rise to an oriented graph (rewriting system)
$\Gamma = \Gamma (\mathcal X,S)$ defined as follows.
The vertices of $\Gamma $ are polynomials from $\Bbbk \langle \mathcal X\rangle $,
two vertices $h_1$, $h_2$ are connected by an edge $h_1\to h_2$
if and only if $h_2$ can be obtained from $h_1$ by reduction (elimination of a principal
part) modulo $S$, i.e., $h_1$ contains a summand $\alpha u$
($\alpha \in \Bbbk \setminus \{0\}$,
$u\in \mathcal X^*$) such that $u = w_1\bar fw_2$ for some words $w_1,w_2\in \mathcal X^*$
and for some $f\in S$,
and $h_2 = h_1 -\alpha w_1fw_2$.

Suppose $\Gamma = \Gamma (\mathcal X,S)$ has the following properties (GSB conditions).
\begin{itemize}
 \item There are no infinite oriented paths in  $\Gamma $. This {\em termination condition}
 says that every polynomial $h\in \Bbbk \langle \mathcal X\rangle $
 can be reduced in a finite number of steps to a terminal form which cannot be further reduced.

 \item If $h\to h_1$ and $h\to h_2$ are two edges in $\Gamma $ then there exist paths
 $h_1\to \dots \to g$ and $h_2\to \dots \to g$ for some vertex~$g$.
This {\em Diamond condition} guarantees uniqueness of the terminal form for each vertex
in~$\Gamma $.
\end{itemize}
Then $S$ is said to be a {\em Gr\"obner--Shirshov basis} in $\Bbbk \langle \mathcal X\rangle $
and those words that are of terminal form (called {\em reduced} words)
form a linear basis
of the associative algebra $\Bbbk \langle \mathcal X\mid S\rangle
= \Bbbk \langle \mathcal X\rangle /(S)$.

It is enough to check the Diamond condition for the following pairs of edges (convergence of ``forks'')
in $\Gamma (\mathcal X,S)$:
\begin{enumerate}
 \item[(C1)] $h = \bar f_1 = w_1\bar f_2 w_2$, $f_1,f_2\in S$, $w_1,w_2\in \mathcal X^*$; then
 $h_1 = r(f_1)$, $h_2 = w_1 r(f_2) w_2$;
 \item[(C2)] $h = \bar f_1u = v\bar f_2$, $f_1,f_2\in S$, $u,v\in \mathcal X^*$, and
 $\deg \bar f_1+\deg \bar f_2 >\deg h$ (so that the principal subwords of $f_1$ and $f_2$ intersect in $h$);
 then $h_1 = r(f_1)u$, $h_2=vr(f_2)$.
\end{enumerate}
In both cases, if there are paths $h\to h_1\to \dots \to g_1$ and $h\to h_2\to \dots \to g_2$, where
$g_1\neq g_2$ are terminal vertices, then we have to add $g_1-g_2$ to the set $S$,
choose its principal part and check the GSB conditions again
for the extended set of relations.

\begin{remark}
 It is convenient if a well order $\leq $ is fixed on $\mathcal X^*$ which is compatible with
 multiplication: then the termination condition is always satisfied
 if $\bar f$ is the principal word if $f$ relative to~$\leq $.
However, it is not necessary to have such an order to find a GSB.
\end{remark}

\begin{example}\label{exmp:Weyl-AlgebraCoeff}
Let $S$ be the set of defining relations
of the algebra $\mathcal A(X,N)$ from  Example \ref{exmp:WeylVertex}
with $\mathcal X = \{T, x(n), y(m) \mid n,m\in \mathbb Z \}$.
Let us write each $f\in S$ in the form $\bar f \to r(f)$ choosing the principal parts as follows:
\begin{gather}
 Tx(n) \to x(n)T - n x(n-1), \quad Ty(n) \to y(n)T - ny(n-1);
   \label{eq:TrWeyl}\\
 x(n)x(m) \to x(m)x(n), \quad y(n)y(m)\to y(m)y(n),\quad n>m;
   \label{eq:CommWeyl}\\
 x(n)y(m) \to y(m)x(n) + [x(n-1),y(m+1)], \quad n>m+1;
   \label{eq:LocXYWeyl}\\
 y(m)x(n) \to x(n)y(m) + [y(m-1),x(n+1)], \quad m\ge n+1.
  \label{eq:LocYXWeyl}
\end{gather}  
This is straightforward to check that the conditions (C1) and (C2) hold for these
relations. Namely, there are no forks of type (C1), and all forks of type (C2) converge.
The key computation of convergence for locality relations 
\eqref{eq:CommWeyl}--\eqref{eq:LocYXWeyl}
in the general form \eqref{eq:FreeLoc}
was done in \cite{Roit1999}, adding \eqref{eq:TrWeyl} obviously does not affect the convergence
since the relations \eqref{eq:FreeLoc} are invariant under the derivation $[T,\cdot ]: z(n)\mapsto -nz(n-1)$,
$z\in X$, $n\in \mathbb Z$.
\end{example}

If $A$ is an  associative algebra generated by a set $\mathcal X$ relative to defining relations~$S$
then a left module~$M$ over~$A$ may be defined also by generators~$Y$ and relations~$R$
(the latter are elements from the free module $\Bbbk \langle \mathcal X\rangle \otimes \Bbbk Y$;
assume a principal part is chosen in each polynomial from~$R$ as well as it was done for~$S$).

Then the split null extension $A\oplus M$ with multiplication
\[
 (a+u)(b+v) = ab + av, \quad a,b\in A,\ u,v\in M,
\]
may be presented as an associative algebra generated by $\mathcal X\cup Y$
relative to the relations $\Sigma = S\cup R \cup \{yz \mid y\in Y,\ z\in \mathcal X\cup Y \}$.
In order to construct a rewriting graph
$\Gamma (\mathcal X\cup Y, \Sigma ) $ for the module $M$ it is enough to consider only those
polynomials (as vertices) from $\Bbbk \langle \mathcal X\cup Y\rangle $ that contain monomials
of the form $uy$, $u\in \mathcal X^*$, $y\in Y$, and check the conditions (C1), (C2).
This leads us to the technique described in \cite{KangLee2000}: how to find
a GSB for a module over an associative algebra.

Namely, assuming $S$ is already a GSB in the associative algebra
$\Bbbk \langle \mathcal X\rangle $, it is enough to check the convergence
of the following ``forks'' $h\to h_1$, $h\to h_2$:
\begin{itemize}
 \item [(CM1)]
 $h = \bar f_1 = w_1\bar f_2w_2y$, $f_1\in R$, $f_2\in S$, $w_1,w_2\in \mathcal X^*$, $y\in Y$;
 then $h_1 = r(f_1)$, $h_2 = w_1r(f_2)w_2y$;
 \item [(CM2)]
 $h = w_1\bar f_1 = \bar f_2w_2y$, $f_1\in R$, $f_2\in S$, $w_1,w_2\in \mathcal X^*$, $y\in Y$,
 $\deg \bar f_1 + \deg\bar f_2 >\deg h$; then $h_1=w_1r(f_1)$, $h_2=r(f_2)w_2y$;
 \item [(CM3)]
 $h = \bar f_1 = w_1\bar f_2$, $f_1,f_2\in R$, $w_1\in \mathcal X^*$; then
 $h_1 = r(f_1)$, $h_2 = w_1 r(f_2)$.
\end{itemize}

\begin{definition}
 Let $V$ be a vertex algebra generated by a set $X$ with locality function
 $N: X\times X \to \mathbb Z_+$, and let
 $A=\mathcal A(X,N)$ be the associative algebra defined above.
A~set of defining relations $R$ of $V$ is said to be a
Gr\"obner--Shirshov basis for the vertex algebra $V$ if $R$
along with \eqref{eq:Vacuum} gives rise to
a Gr\"obner--Shirshov basis of the 1-generated $A$-module (i.e., $Y=\{\mathbf 1\}$).
\end{definition}

\begin{example}\label{exmp:FreeDerCom}
 Let $V$ be the vertex algebra generated by a single element $v$
such that $N(v,v)=0$, with no more relations.
\end{example}

The GSB of $A=\mathcal A(X,N)$ 
consists of $v(n)v(m)\to v(m)v(n)$, $n>m$,
and $Tv(n)\to v(n)T - nv(n-1)$.
Then the locality and vacuum relations form a Gr\"obner--Shirshov basis of $V$
as of an $A$-module:
\[
 v(n)\mathbf 1 \to 0,\ n\ge 0,\quad T\mathbf 1\to 0.
\]
The terminal words that form a linear basis of $V$ are
\begin{equation}\label{eq:Normal-1-generated}
 v(-n_1)\dots v(-n_k)\mathbf 1,\quad n_1\ge \dots \ge n_k>0,\ k\ge 0.
\end{equation}
One may easily recognize here the free commutative differential algebra in one variable $v = v(-1)\mathbf 1$, $T$ is the derivation.

\begin{example}\label{exmp:Abelian}
 Let $V$ be the vertex algebra generated by one element $e$
 with one defining relation $Te=0$.
\end{example}

Then $(e\oo\lambda e)=0$ by sesquilinearity, so the corresponding algebra
$A=\mathcal A(X,N)=\mathcal A(\{e\}, 0)$
is similar to the algebra from Example~\ref{exmp:FreeDerCom}.
The $A$-module $V$ is generated by $\mathbf 1$ relative to the defining relations
\[
 e(n)\mathbf 1 = Te(-1)\mathbf 1 = T\mathbf 1= 0,\quad n\ge 0.
\]
The GSB of $A$ is given by the relations (written as rewriting rules)
\[
 e(n)e(m)\to e(m)e(n),\ n>m,\quad Te(n)\to e(n)T - n e(n-1),\ n\in \mathbb Z,
\]
and in order to get a GSB of $V$ it is enough to add the relations
\[
 e(-n)\underbrace{e(-1)\dots e(-1)}_{l}\mathbf 1\to 0, \ n>1,\ l\ge 0.
\]
For $l=0$, they appear from forks like 
$Te(-1)\mathbf 1 \to 0$ 
$Te(-1)\mathbf 1 \to e(-2)\mathbf 1 + e(-1)T\mathbf 1$
of type (CM3), for $l>0$, forks of type (CM2) provide the desired relations.

The terminal words are of the form $e(-1)^m\mathbf 1$, $m\ge 0$, so $V$
is the polynomial algebra in one variable with trivial derivation $T$
and trivial $\lambda $-bracket.

Let us consider in more details how to find a GSB
for a more complicated example, namely, for the vertex algebra from  Example \ref{exmp:WeylVertex}.

\begin{example}
Let $A$ be the algebra $\mathcal A(X,N)$
with $X=\{x,y\}$, $N(x,x)=N(y,y)=0$, $N(x,y)=N(y,x)=1$,
and let $Y = \{\mathbf 1\}$.
Then the vertex algebra $W=\mathrm{Vert}(X,N\mid (x\oo0y)=\mathbf 1)$
is isomorphic to the 1-generated $A$-module relative to the relations
\begin{gather}
x(n)\mathbf 1 \to 0, \quad y(n)\mathbf 1 \to 0, \quad n\ge 0, \label{eq:WeylLoc}\\
T\mathbf 1 \to 0, \label{eq:WeylVacuum}\\
x(0)y(-1)\mathbf 1 \to  \mathbf 1. \label{eq:WeylComm0}
\end{gather}
\end{example}

Let us determine the structure of the $A$-module $W$ defined as above.
Note that 
for every $a\in W$ we have
$x(n)a=y(n)a=0$ in $W$ for all sufficiently large~$n$.
This property represents locality in a vertex algebra, it also can be derived from 
\eqref{eq:LocXYWeyl}, \eqref{eq:LocYXWeyl}, and \eqref{eq:WeylLoc}.

There is a plenty of non-converging ``forks'' of all types (CM1)--(CM3)
for the edges defined by these rules
along with the rules from Example \ref{exmp:Weyl-AlgebraCoeff}.

First, let us consider the vertex $Tx(0)y(-1)\mathbf 1$.
On the one hand,
$Tx(0)y(-1)\mathbf 1\to x(0)Ty(-1)\mathbf 1 = -x(0)y(-2)\mathbf 1 + x(0)y(-1)T\mathbf 1
\to -x(0)y(-2)\mathbf 1$
by \eqref{eq:TrWeyl} and \eqref{eq:WeylVacuum}.
On the other hand,
$Tx(0)y(-1)\mathbf 1 \to T\mathbf 1 \to 0$
by \eqref{eq:WeylComm0}. Hence, we have to add the rule
$x(0)y(-2)\mathbf 1 \to 0$ to the set of defining relations.
Similarly,
\begin{equation}\label{eq:WeylComm0-n}
x(0)y(-n)\mathbf 1 \to 0 , \quad n\ge 2.
\end{equation}

Relations \eqref{eq:WeylComm0-n} modulo \eqref{eq:LocXYWeyl} and \eqref{eq:WeylLoc} imply
\begin{equation}\label{eq:CommJacWeyl-triv}
 [x(n),y(m)]\mathbf 1 \to \delta_{n+m,-1}\mathbf 1, \quad n,m\in \mathbb Z.
\end{equation}
More generally, we have the following

\begin{proposition}
If $A$ is the algebra $\mathcal A(X,N)$ from Example~\ref{exmp:WeylVertex} and
$W$ is the $A$-module generated by a single element $\mathbf 1$ modulo
the defining relations \eqref{eq:WeylLoc}--\eqref{eq:WeylComm0}
then for every word $u$ (including the empty word)
in the alphabet $\{x(n),y(n)\mid n\in \mathbb Z\}$
the following relations hold in $W$:
\begin{gather}
y(m)x(m)u\mathbf 1 \to x(m)y(m)u\mathbf 1, \quad m\in \mathbb Z;
 \label{eq:Rel1-yx}\\
x(m+1)y(m)u\mathbf 1 \to y(m)x(m+1)u\mathbf 1, \quad m\in \mathbb Z,\ m\ne -1;
 \label{eq:Rel1-xy} \\
x(0)y(-1)u\mathbf 1 \to y(-1)x(0)u\mathbf 1 + u\mathbf 1.
 \label{eq:Rel1-xy0}
\end{gather}
\end{proposition}

\begin{proof}
Let us prove by induction on the length of $u$ that 
\begin{equation}\label{eq:CommRelWeyl}
x(n)y(m)u\mathbf 1 = y(m)x(n)u\mathbf 1 + \delta_{n+m,-1}u\mathbf 1.
\end{equation}
For the empty word $u$ it follows from \eqref{eq:CommJacWeyl-triv}.
Suppose $u=x(k)v$, where $v$ is a shorter word
(the case $u=y(k)v$ is analogous). 
Then the inductive assumption along with the defining relations of $A$
(see Example~\ref{exmp:WeylVertex}) imply
\begin{multline*}
x(n)y(m)u\mathbf 1 
= x(n)y(m)x(k)v \mathbf 1 
= x(n)x(k)y(m)v\mathbf 1 -\delta_{m+k,-1}x(n)v\mathbf 1 \\
=x(k)x(n)y(m)v\mathbf 1 -\delta_{m+k,-1}x(n)v\mathbf 1 \\
= x(k)y(m)x(n)v\mathbf 1 +\delta_{n+m,-1}x(k)v\mathbf 1 
-\delta_{m+k,-1}x(n)v\mathbf 1 \\
= y(m)x(k)x(n)v\mathbf 1 +[x(k),y(m)]x(n)v\mathbf 1+\delta_{n+m,-1}x(k)v\mathbf 1 
-\delta_{m+k,-1}x(n)v\mathbf 1.
\end{multline*}
Note that 
\[
[x(k),y(m)]x(n)v\mathbf 1 = [x(k+N),y(m-N)]x(n)v\mathbf 1 = x(k+N)y(m-N)x(n)v\mathbf 1
\]
for sufficiently large $N$. Apply the inductive assumption again to get 
\begin{multline*}
x(k+N)y(m-N)x(n)v\mathbf 1 = x(k+N)x(n)y(m-N)v\mathbf 1 \\
= x(n)x(k+N)y(m-N)v\mathbf 1
= x(n)[x(k+N),y(m-N)]v\mathbf 1 = \delta_{k+m,-1} x(n)v\mathbf 1.
\end{multline*}
Hence,
\[
 x(n)y(m)x(k)v \mathbf 1 = y(m)x(k)x(n)v\mathbf 1 +\delta_{n+m,-1}x(k)v\mathbf 1
\]
and it remains to switch $x(n)$ and $x(k)$ to get the desired relation.
\end{proof}

Let $\Sigma $ stand for the system of relations
\eqref{eq:TrWeyl}--\eqref{eq:LocYXWeyl}
 along with
 \eqref{eq:WeylLoc}--\eqref{eq:WeylComm0}
 and \eqref{eq:Rel1-yx}--\eqref{eq:Rel1-xy0}.
This system meets the termination condition, and
the terminal words in $\Gamma(X\cup\{\mathbf 1\}, \Sigma)$
(those that correspond to elements of $W$)
are of the form
\begin{equation}\label{eq:Weyl-basis}
 u\mathbf 1, \quad u = z_1(n_1)z_2(n_2)\ldots z_k(n_k), \quad z_i\in \{x,y\},
\end{equation}
where $n_1\le n_2\le \dots \le n_k<0$ ($k\ge 0$); if $z_i=y$ and $z_{i+1}=x$ then
$n_i<n_{i+1}$.

\begin{theorem}\label{thm:WeylGSB}
The relations
 \eqref{eq:WeylLoc}--\eqref{eq:WeylComm0},
and \eqref{eq:Rel1-yx}--\eqref{eq:Rel1-xy0} with terminal words $u$ as in \eqref{eq:Weyl-basis}
form a GSB of the vertex algebra $W=\mathrm{Vert}(X,N\mid (x\oo0 y)=\mathbf 1)$.
\end{theorem}

\begin{proof}
The Diamond Condition for
\eqref{eq:TrWeyl}--\eqref{eq:LocYXWeyl}, \eqref{eq:WeylLoc}--\eqref{eq:WeylComm0},
and \eqref{eq:Rel1-yx}--\eqref{eq:Rel1-xy0}
may be checked in a straightforward way.
 Let us list all ambiguities (``forks'') of types (CM1)--(CM3) for which the
 convergence  should be checked:
\[
\begin{gathered}
 y(n)x(m)\mathbf 1, \ n>m\ge 0,
 \quad
 x(n)y(m)\mathbf 1,\ n>m+1 \ge 1,
 \quad
 Tx(n)\mathbf 1,\ Ty(n)\mathbf 1,\ n\ge 0, \\
 x(n)x(m)\mathbf 1, \ y(n)y(m)\mathbf 1, \ n>m\ge 0,
 \quad
 x(m+1)y(m)\mathbf 1, \ y(m)x(m)\mathbf 1,\ m\ge 0, \\
y(n)x(m+1)y(m)u\mathbf 1, \ n\ge m+2,
\quad
x(n)y(m)x(m)u\mathbf 1, \  n>m+1, \\
x(n+1)y(n)x(m)u\mathbf 1,\ x(n+1)y(n)y(m)u\mathbf 1,\ n\ge  m+1, \\
\quad
y(n)x(n)y(m)u\mathbf 1, \ n>m+1,
\quad
y(n)x(n)x(m)u\mathbf 1,\ n>m, \\
x(n)x(m+1)y(m)\mathbf 1,\ n>m+1,
\quad
y(n)y(m)x(m)u\mathbf 1,\ n>m, \\
T x(m+1)y(m)u\mathbf 1,\quad Ty(m)x(m)u\mathbf 1.
\end{gathered}
\]
For example, choose $h = y(n)x(n)y(m)u\mathbf 1$, $n>m+1$.
On the one hand,
$h\to h_1 = x(n)y(n)y(m)u\mathbf 1$
by \eqref{eq:Rel1-yx}.
On the other hand,
$h \to h_2 = y(n)y(m)x(n)u\mathbf 1 + y(n)[x(n-1),y(m+1)]u\mathbf 1$
by \eqref{eq:LocXYWeyl}.
Then
\[
 h_1\to x(n)y(m)y(n)u\mathbf 1\to
g = y(m)x(n)y(n)u\mathbf 1 + [x(n-1),y(m+1)]y(n)u\mathbf 1,
\]
\begin{multline*}
 h_2\to y(m)y(n)x(n)u\mathbf 1 + y(n)x(n-1)y(m+1)u\mathbf 1- y(n)y(m+1)x(n-1)u\mathbf 1 \\
\to y(m)x(n)y(n)u\mathbf 1 + x(n-1)y(n)y(m+1)u\mathbf 1
+ [y(n-1),x(n)]y(m+1)u\mathbf 1 \\
- y(m+1)y(n)x(n-1)u\mathbf 1
\to
y(m)x(n)y(n)u\mathbf 1 + x(n-1)y(m+1)y(n)u\mathbf 1 \\
+ [y(n-1),x(n)]y(m+1)u\mathbf 1
- y(m+1)x(n-1)y(n)u\mathbf 1 - y(m+1)[y(n-1),x(n)]u\mathbf 1\\
=
y(m)x(n)y(n)u\mathbf 1 + [x(n-1),y(m+1)]y(n)u\mathbf 1 \\
+ [y(n-1),x(n)]y(m+1)u\mathbf 1 - y(m+1)[y(n-1),x(n)]u\mathbf 1\\
\end{multline*}
The last two summands give zero by \eqref{eq:Rel1-xy} or \eqref{eq:Rel1-xy0}, so the
type (CM2) fork $h\to h_1$, $h\to h_2$ converges to $g$.

Let us also consider an example of a type (CM1) fork with $h = Tx(m+1)y(m)u\mathbf 1$.
On the one hand,
$h\to h_1 = x(m+1)Ty(m)u\mathbf 1 - (m+1)x(m)y(m)u\mathbf 1$.
On the other hand,
$h\to h_2 = Ty(m)x(m+1)u\mathbf 1 + \delta_{m,-1}Tu\mathbf 1$.
Note that $Tu\mathbf 1\to \dots \to u'\mathbf 1$, where $u'$ is a linear combination
of words obtained from the derivation $[T,u]$ in $A$.
Then
\begin{multline*}
 h_1
 \to x(m+1)y(m)u'\mathbf 1 -m x(m+1)y(m-1)u\mathbf 1 - (m+1)x(m)y(m)u\mathbf 1 \\
 \to y(m)x(m+1)u'\mathbf 1 +\delta_{m,-1}u'\mathbf 1 -m( y(m-1)x(m+1)u\mathbf 1 + [x(m),y(m)]u\mathbf 1) \\
 - (m+1)x(m)y(m)u\mathbf 1
 \to g  =
y(m)x(m+1)u'\mathbf 1 +\delta_{m,-1}u'\mathbf 1 \\ -my(m-1)x(m+1)u\mathbf 1
- (m+1)x(m)y(m)u\mathbf 1,
\end{multline*}
\begin{multline*}
 h_2\to
 y(m)x(m+1)u'\mathbf 1 -m y(m-1)x(m+1)u\mathbf 1 - (m+1)y(m)x(m)u\mathbf 1 \\
 + \delta_{m,-1}Tu\mathbf 1 \to g
\end{multline*}
as above. Hence, this fork also converges.
The convergence of all other forks can be checked similarly.
\end{proof}

\begin{corollary}\label{cor:Weyl-basis}
 A linear basis of the vertex algebra from Example~\ref{exmp:WeylVertex}
 is given by the terminal words \eqref{eq:Weyl-basis}.
\end{corollary}

\begin{remark}
 Modulo the relations \eqref{eq:CommRelWeyl}, the basis \eqref{eq:Weyl-basis}
 may be replaced with the set of words
\[
 x(-n_1)\dots x(-n_r)y(-m_1)\dots y(-m_k)\mathbf 1,
\]
where $n_1\ge \dots \ge n_r>0$, $m_1\ge \dots \ge m_k>0$.
The latter is known to be the basis of the Weyl vertex algebra \cite{AdamPedic19} which meets the relation
$(x\oo\lambda y) =\mathbf 1$,
so Example~\ref{exmp:WeylVertex} actually describes this vertex algebra.
\end{remark}

\section{Universal vertex envelopes of Lie conformal algebras}

As follows from the definition, there is a forgetful functor
from the category of vertex algebras to the category of Lie conformal algebras
(somewhat similar to the forgetful functor from the category of Poisson algebras
to that of Lie algebras). There exists its left adjoint functor
$U_{\mathrm {Vert}}$ that turns a Lie conformal algebra to its
universal vertex envelope. The structure of such a vertex algebra was described in
\cite{Roit1999}. It is also easy to state in terms of generators and relations.

Suppose $L$ is a Lie conformal algebra generated by a set $X$.
Assume the function $N$ is the restriction of the locality function $N_L$ onto $X\times X$.
In general, the structure of $L$ is determined
by defining relations $R$ between generators stated in terms of the operations
$T$ and $(\cdot\oo n \cdot )$, $n\ge 0$.
By the axioms of a Lie conformal algebra, each relation may be presented
as a $\Bbbk [T]$-linear combination of right-normed words
\[
 (a_1\oo {n_1} (a_2\oo{n_2} \dots (a_k\oo{n_k} a_{k+1})\dots )), \quad a_i\in X,\ n_i\ge 0.
\]
Then $U_{\mathrm {Vert}}(L) = \mathrm{Vert}(X,N\mid R)$,
where the elements of $R$ are interpreted as the same $\Bbbk[T]$-linear combinations
of
\[
 a_1(n_1)\dots a_{k}(n_k)a_{k+1}(-1)\mathbf 1.
\]

Many classical examples of vertex algebras are quotients of universal vertex envelopes
of Lie conformal algebras.
Let $E$ be a 1-dimensional central extension of a Lie conformal algebra $L$
with a $\Bbbk [T]$-basis $X$, that is, 
we are given an exact sequence of conformal algebras
\begin{equation}\label{eq:CentralExt}
 0\to \Bbbk e \to E \to L\to 0, \quad Te=0.
\end{equation}

The structure of $E$ is completely determined
by its multiplication table with polynomial coefficients
\[
 (x\oo\lambda y ) = \sum\limits_{z\in X} h_{x,y}^z(T,\lambda )z + \varphi _{x,y}(\lambda )e,
\]
for $x,y\in X$.
Assume $X$ is totally ordered in some way, and $N: X\times X\to \mathbb Z_+$ is the locality
function in $E$, i.e.,
$N(x,y)  = \max \{\deg \varphi_{x,y}, \deg_\lambda h_{x,y}^z \mid z\in X \} + 1$.
Since $e$ is central, set $N(e,x)=N(x,e)=N(e,e)=0$.
Then
\[
 U_{\mathrm{Vert}}(E) = \mathrm{Vert} (X\cup\{e\}, N \mid R),
\]
where $R$ consists of the relations
\begin{gather*}
x(n)e(-1)\mathbf 1\to 0,\ e(n)x(-1)\mathbf 1\to 0,\ e(n)e(-1)\mathbf 1,\quad n\ge 0, \\
Te(-1)\mathbf 1 \to 0, \\
 x(n)y(-1)\mathbf 1 \to \sum\limits_{z\in X } h_{x,y}^{z,n}(T)z(-1)\mathbf 1 + \varphi _{x,y}^n e(-1)\mathbf 1, \quad x,y\in X,\ n\ge 0.
\end{gather*}
Here
$h_{x,y}^{z,n}(T)$ and $\varphi_{x,y}^n$ are the coefficients at $\lambda ^n/n!$
of the poly\-no\-mials $h_{x,y}^z(T,\lambda )$ and 
$\varphi _{x,y}(\lambda )$, respectively.

Let us order the set $\mathcal X = \{T, x(n), e(m) \mid x\in X, n,m\in \mathbb Z\}$
of generators of $\mathcal A(X\cup\{e\}, N)$ as follows:
$T>e(n)>x(m)$ for all $x\in X$, $n,m\in \mathbb Z$, for $x,y\in X$ let $y(n)>x(m)$ if and only
if $(n,y)>(m,x)$ lexicographically. It is convenient to order $e(n)$ as follows:
\[
 e(-1)<e(0)<e(1)<\dots <e(-2)<e(-3)<\dots .
\]

\begin{theorem}[c.f. PBW-Theorem for vertex algebras \cite{Roit1999}]\label{thm:VertexPBW}
The Gr\"obner--Shirshov basis of $U_{\mathrm{Vert}}(E)$
as of $\mathcal A(X\cup \{e\},N)$-module 
consists of
\[
\begin{aligned}
x(n)e(-1)^l\mathbf 1 & \to 0,\ n\ge 0,\ x\in X,\ l\ge 0, \\
T\mathbf 1& \to 0, \quad  e(n)\mathbf 1\to 0,\ n\ne -1, \\
x(n)y(m)u\mathbf 1 & \to y(m)x(n)u\mathbf 1 \\
 & + \sum\limits_{z\in X}\sum\limits_{k,s\ge 0} 
  (-1)^k\binom{n}{s}\binom{n+m-s}{k} \alpha^{z,n}_{x,y,k} z(n+m-s-k)u\mathbf 1  \\
& + \sum\limits_{s\ge 0} \binom{n}{s} \varphi_{x,y}^n e(n+m-s)u\mathbf 1, \\
  & x(n)>y(m),\ x,y\in X,
\end{aligned}
\]
where $u$ is a word in the alphabet $\mathcal X\setminus \{T\}$,
$\alpha^{z,n}_{x,y,k}$ are coefficients at $T^k/k!$ of $h_{x,y}^{z,n}(T)$.
\end{theorem}

This statement can be checked similarly to Theorem~\ref{thm:WeylGSB},
as well as derived from the construction of
$ U_{\mathrm{Vert}}(L)$
in terms of coefficient algebras in \cite{Roit1999}.
Namely, the last rewriting rule represents 
the equality 
\[
[x(n),y(m)] = \sum\limits_{s\ge 0 } \binom{n}{s} (x\oo s y)(n+m-s)
\]
in the coefficient Lie algebra $\mathcal L(E)$ of $E$.
The terminal words that form a linear basis of
$ U_{\mathrm{Vert}}(L)$
are of the form
\[
 x_1(n_1)\dots x_k(n_k)\underbrace{e(-1)\dots e(-1)}_l\mathbf 1,
 \quad x_i\in X, \ x_1(n_1)\le \dots x_k(n_k),\ n_i<0,\ l,k\ge 0.
\]
This is exactly the linear basis of 
$U(\mathcal L(E))\otimes _{U_+} \Bbbk \mathbf 1$, 
where $U_+$ stands for the subalgebra in 
$U(\mathcal L(E))$ generated by $x(n)$, $n\ge 0$, $x\in X$.
The latter induced module is exactly $U_{\mathrm{Vert}}(E)$
considered in \cite{Roit1999}.

\begin{example}
 Let $L$ be the Virasoro conformal algebra from Example \ref{exmp:VirConformal}
 denoted $\mathrm{Vir }$.
It has a non-trivial 1-dimensional central extension $E=\mathrm{Vir}_c$
which is generated by $v$ and $e$, where $Te=0$, $e$ is central, and
\[
 (v\oo\lambda v) = (T+2\lambda )v + \dfrac{1}{12} c\lambda^3 e,
\]
$c\in \Bbbk $.
\end{example}

The quotient of the universal vertex envelope $U_{\mathrm{Vert}}(\mathrm{Vir}_c)$
modulo the ideal generated by $e(-1)\mathbf 1 -\mathbf 1$ is known as
the {\em Virasoro vertex algebra of central charge~$c$} (see, e.g., \cite{F-BZvi}).
The rule $e(-1)\mathbf 1\to \mathbf 1$ does not give birth to new non-converging forks.
Hence,
the linear basis of the Virasoro vertex algebra consists of the words
\eqref{eq:Normal-1-generated},  as in the classical PBW-Theorem.

\begin{example}
 Let $L$ be 1-generated Abelian Lie conformal algebra, i.e.,
 $L=\Bbbk [T]v$ with $(v\oo\lambda v)=0$.
 Then for every odd polynomial $f=f(\lambda )\in \Bbbk [\lambda ]$
 such that $f(-\lambda )=-f(\lambda )$
there is a non-trivial 1-dimensional central extension of $L$,
a Lie conformal algebra
$E = H_f = \Bbbk [T]v+\Bbbk e$, $Te=0$, where
\[
 (v \oo\lambda v ) = f(\lambda )e.
\]
\end{example}

For $f(\lambda )=\lambda $, the quotient of $U_{\mathrm{Vert}}(H_\lambda )$ modulo
the ideal generated by $e(-1)\mathbf 1 -\mathbf 1$
is known as the {\em Heisenberg vertex algebra}.

From the GSB point of view, there is no much difference
between the vertex envelopes of $H_\lambda $ and $H_f$ for an arbitrary odd
polynomial~$f$.
In particular, the linear basis of
$U_{\mathrm{Vert}}(H_\lambda )/ (e(-1)\mathbf 1 -\mathbf 1)$
consists of the words \eqref{eq:Normal-1-generated}.

Note that the Weyl vertex algebra is also a quotient of the 1-dimensional
central extension 
\eqref{eq:CentralExt}
of rank two Abelian Lie conformal algebra $L$ spanned by $x$, $y$
so that $\varphi_{x,x}=\varphi_{y,y}=0$, $\varphi_{x,y}=-\varphi_{y,x}=1$.
Then $(x\oo 0 y)= e$ in $E$, thus $U_{\mathrm {Vert}}(E)/(e-\mathbf 1)$ is the Weyl vertex algebra.

In order to compute GSB in a more complicated example, 
let us modify Example~\ref{exmp:WeylVertex} as follows.

\begin{example}\label{exmp:CommVertex}
Let $X=\{x,y\}$, $N(x,x)=N(y,y)=0$, $N(x,y)=N(y,x)=1$, 
and $R=\{x.y-y.x-\mathbf 1\}$.
\end{example}

In order to find GSB of $V = \mathrm {Vert}(X,N\mid R)$
note that $x.y-y.x = T(x\oo 0 y)$ by \eqref{eq:VertexComm}.
Let us add new generator $e=(x\oo{0} y)$ such that $T^2e=0$,
and consider central extension 
\[
0\to \Bbbk e+\Bbbk Te \to E \to L\to 0,
\]
where $L$ is the Abelian Lie conformal algebra freely generated over $\Bbbk [T]$ by $x$, $y$. Hence, $V\cong U_{\mathrm{Vert}}(E)/(Te-\mathbf 1)$.

The GSB of $U_{\mathrm{Vert}}(E)$ is easy to find 
(in a similar way as in Theorem~\ref{thm:VertexPBW})
with the same ordering of generators $\mathcal X$.
The linear basis of this vertex envelope 
may be presented as
\begin{equation}\label{eq:U-basis}
u e(-1)^l e(-2)^m\mathbf 1,
\end{equation}
where $u=z_1(n_1)\dots z_k(n_k)$ as in \eqref{eq:Weyl-basis}, $k,l,m\ge 0$.
The new relation $Te=\mathbf 1$
represented as $e(-2)\mathbf 1 \to \mathbf 1$ 
does not produce new non-converging forks, so the basis of $V$
in Example \ref{exmp:CommVertex} consists of the same words 
\eqref{eq:U-basis}
with $m=0$.

\begin{remark}
 The left-symmetric 
 subalgebra generated in $V$ from
 Example~\ref{exmp:CommVertex} 
 by the set $X$ relative to the 
 operation $(\cdot . \cdot)$ is not associative.

Indeed, for example, the associator 
$(x.x).y - x.(x.y)$ is equal to 
\[
y(-1)x(-1)x(-1)\mathbf 1 - Ty(0)x(-1)x(-1)\mathbf 1 - x(-1)x(-1)y(-1)\mathbf 1
\]
by \eqref{eq:VertexComm} and \eqref{eq:WickIdent}.
Since $y(-1)x(-1)u\mathbf 1 = x(-1)y(-1)u\mathbf 1 - e(-2)u\mathbf 1$
and $y(0)x(-1)u\mathbf 1 = x(-1)y(0)u\mathbf 1 -e(-1)u\mathbf 1$,
we get 
$(x.x).y - x.(x.y) = 2x(-2)e(-1)\mathbf 1 \ne 0$.
\end{remark}

\section{Vertex envelopes of left-symmetric algebras}

Let $\mathrm{Vert} $ be the category of vertex algebras, 
and let $\mathrm{LieConf}$ and $\mathrm{LSym}$ stand for the categories 
of Lie conformal and left-symmetric algebras, respectively. 
We have already considered the forgetful functor
$ \mathrm{Vert} \to \mathrm{LieConf}$
whose left adjoint functor is $U_{\mathrm{Vert}}(\cdot )$.
As follows from the definition, 
there is also a functor 
\[
\Psi : \mathrm{Vert}\to \mathrm{LSym }
\]
ignoring the $\lambda $-bracket on a vertex algebra.
Let us observe the properties of the functor~$\Psi $.
Recall that left-symmetric algebras are Lie-admissible, 
i.e., the operation $[a,b] = a.b-b.a$ turns $A\in \mathrm{LSym}$ into a Lie algebra $A^{(-)}$.

\begin{proposition}\label{prop:LieNilpA}
Let $A$ be a left-symmetric algebra embeddable into 
$\Psi (V)$ for a vertex algebra~$V$, 
and let there exists $N\in \mathbb Z_+$ such that 
$N_V(a,b)\le N$  for all $a,b\in A$.
Then $A$ is Lie-nilpotent.
\end{proposition}

The latter condition means the commutator Lie algebra $A^{(-)}$
is nilpotent.

\begin{proof}
Assume $A$ embeds into $\Psi(V)$ for an appropriate vertex algebra~$V$. If $(a\oo\lambda b) \in \lambda^k V[\lambda ]$
for some $a,b\in A$, $k\in \mathbb Z_+$, then \eqref{eq:VertexComm} 
implies
\[
[a,b] = a.b - b.a \in T^{k+1} V.
\]
Hence, $([a,b]\oo\lambda A)\in \lambda^{k+1}V[\lambda ]$. 
Therefore, 
$A^{(n)} = [\dots [[A,A],A], \dots , A]\subset A$ has the following 
property: for every $a\in A^{(n)}$ and for every $b\in A$
$\deg_\lambda (a\oo\lambda b)\ge n+2$. Since the degree 
of the $\lambda $-bracket 
is uniformly bounded by a constant $N-1$, we have $A^{(N+1)} = 0$.
\end{proof}

\begin{remark}
Proposition \ref{prop:LieNilpA} and its proof are easy to 
generalize for superalgebras.
\end{remark}

For example, if a finite-dimensional left-symmetric algebra
$A$ embeds into a vertex algebra then the conditions 
of Proposition \ref{prop:LieNilpA} hold.
It was pointed out in \cite{Burde} that neither simple left-symmetric algebra can be Lie-nilpotent. Hence we obtain

\begin{corollary}
A simple finite-dimensional left-symmetric algebra 
cannot be embedded into a vertex algebra.
\end{corollary}

An associative and commutative algebra $A$
(with abelian $A^{(-)}$)
obviously embeds into a vertex algebra: 
it is enough to define identically zero $\lambda $-bracket,
i.e., set $N=N(a,b)=0$ for all $a,b\in A$.
In this case, a vertex algebra turns into 
just differential commutative algebra.

It is reasonable to suspect that 3-nilpotence of 
the commutator Lie algebra 
$A^{(-)}$ for a left-symmetric algebra $A$
is enough to claim  an 
embedding of $A$ into a vertex algebra $V$ such that 
$N_V(A,A)\le 1$. 
In the rest of this section, we show 
that this is not true. 

Assume $A$ is a left-symmetric algebra with a linear basis 
$X$, and $A^{(-)}$ is 3-nilpotent. 
Then 
$V(A,1) = \mathrm{Vert}(X, N=1\mid x.y-xy, \ x,y\in X)$
is the universal object in the class of vertex 
envelopes $V$ of $A$ with $N_V(a,b)\le 1$
for all $a,b\in A$.

\begin{lemma}\label{lem:00-null}
For all $x,y,z\in X$ we have $T^2x(0)y(0)z(-1)\mathbf 1 =0$ 
in $V(A,1)$.
\end{lemma}

\begin{proof}
Use \eqref{eq:VertexComm} and 
multiplication rules like
$y(-1)z(-1)\mathbf 1 \to (yz)(-1)\mathbf 1$
to derive
\begin{multline*}
T^2x(0)y(0)z(-1)\mathbf 1 =Tx(0)y(0)z(-2)\mathbf 1 
= Tx(0)[y(-1),z(-1)]\mathbf 1 \\
= Tx(0)[y,z](-1)\mathbf 1 
= [x,[y,z]](-1)\mathbf 1 = 0. 
\end{multline*}
\end{proof}

Let us denote $\circlearrowleft F(a,b,c) = F(a,b,c)+F(b,c,a)+F(c,a,b)$
for an algebraic expression $F$ in three variables.

\begin{proposition}
For all $a,b,c \in A$ we have 
\begin{equation}\label{eq:CircleAss}
\circlearrowleft (a,b,c)_. = 0 
\end{equation}
in $V(A,1)$,
where
$(a,b,c)_. = (a.b).c - a.(b.c)$.
\end{proposition}

\begin{proof}
By definition, 
$\circlearrowleft (a,b,c)_. =(ab).c - a.(bc) + (bc).a-b.(ca)
+(ca).b - c.(ab) = [ab,c] + [bc, a] + [ca, b]$
for all $a,b,c \in A$.
It follows from \eqref{eq:VertexComm}
that 
\[
[ab,c] = -\int\limits_{-T}^0 (c\oo\lambda ab)\,d \lambda 
= -T(c\oo 0 (a.b)).
\]
Apply \eqref{eq:WickIdent} to get 
$(c\oo 0 (a.b)) = a.(c\oo 0 b) + (c\oo 0 a).b$, 
so 
$\circlearrowleft (c\oo 0 (a.b)) 
= a.(c\oo 0 b) + (c\oo 0 a).b
+ b.(a\oo 0 c) + (a\oo 0 b).c
+ c.(b\oo 0 a) + (b\oo 0 c).a
= \circlearrowleft [a,(c\oo 0 b)]$
since $(x\oo 0 y) = -(y\oo 0 x)$
for all $x,y\in A$.

Hence, 
\[
\circlearrowleft (a,b,c)_. = -T \circlearrowleft [a,(c\oo 0 b)]=
-T\circlearrowleft \int\limits_{-T}^0 (a\oo\lambda (c\oo 0 b)).
\]
Note that 
$a(n)c(0)b(-1)\mathbf 1 = 0$ for $n\ge 2$ due to locality relations in the corresponding algebra $\mathcal A(X,1)$.
Therefore, 
$(a\oo\lambda (c\oo 0 b)) = a(0)c(0)b(-1)\mathbf 1 + \lambda
a(1)c(0)b(-1)\mathbf 1$, and 
\begin{multline*}
T\int\limits_{-T}^0 (a\oo\lambda (c\oo 0 b)) 
= T^2 \left (
 a(0)c(0)b(-1)\mathbf 1 - \dfrac{T}{2} a(1)c(0)b(-1)\mathbf 1
\right) \\
= T^2 \left (
\dfrac{3}{2} a(0)c(0)b(-1)\mathbf 1
- \dfrac{1}{2} a(1)c(0)b(-2)\mathbf 1
\right)
\end{multline*}
The first summand gives us zero by Lemma~\ref{lem:00-null}, 
the second one may be represented as 
$a(1)[c,b](-1)\mathbf 1 = 0$ by locality.
Hence, $\circlearrowleft (a,b,c)_. = 0$.
\end{proof}

It remains to note that 
\eqref{eq:CircleAss} is not a corollary of left-symmetric 
identity and 3-nilpotence of the commutator Lie algebra.
For example, let $A$ be a 3-dimensional algebra spanned by 
$x$, $y$, $z$ with the following multiplication table:
\[
xx=x+y,\quad xy=(1-\gamma)z,\quad yx=-\gamma z, \quad yy=\gamma z,
\]
where $\gamma \in \Bbbk $ and 
other products are zero. 
This is a left-symmetric algebra with nilpotent $A^{(-)}$, 
but $\circlearrowleft (x,x,x) = 3[x+y,x] = -3z\ne 0$. 
Hence, $A $ cannot be embedded into a vertex algebra $V$
with $N_V(A,A)\le 1$.

\end{document}